\documentclass[12pt,reqno]{amsart}

\usepackage{amssymb}

\newtheorem{theorem}{Theorem}[section]
\newtheorem{proposition}[theorem]{Proposition}
\newtheorem{lemma}[theorem]{Lemma}
\newtheorem{corollary}[theorem]{Corollary}

\theoremstyle{definition}

\numberwithin{equation}{section}

\begin{document}
\baselineskip=15.5pt

\title[Numerical flatness and principal bundles on Fujiki manifolds]{Numerical
flatness and principal bundles on Fujiki manifolds}

\author[I. Biswas]{Indranil Biswas}

\address{School of Mathematics, Tata Institute of Fundamental
Research, Homi Bhabha Road, Mumbai 400005, India}

\email{indranil@math.tifr.res.in}



\date{}

\begin{abstract}
Let $M$ be a compact connected Fujiki manifold, $G$ a semisimple affine algebraic group
over $\mathbb C$ with one simple factor and $P$ a fixed
proper parabolic subgroup of $G$. For a holomorphic principal $G$--bundle $E_G$ over $M$,
let ${\mathcal E}_P$ be the holomorphic principal $P$--bundle $E_G\, \longrightarrow\, E_G/P$ given by
the quotient map.
We prove that the following three statements are equivalent: (1) ${\rm ad}(E_G)$ is numerically flat,
(2) the holomorphic line bundle $\bigwedge^{\rm top} {\rm ad}({\mathcal E}_P)^*$ is nef, and (3)
for every reduced irreducible compact complex analytic space $Z$ with a K\"ahler form $\omega$,
holomorphic map $\gamma\, :\, Z\, \longrightarrow\, M$, and holomorphic reduction of structure group
$E_P\, \subset\, \gamma^*E_G$ to $P$, the inequality ${\rm degree}({\rm ad}(E_P))\, \leq\, 0$
holds.
\end{abstract}

\maketitle

\section{Introduction}

A basic theorem of Miyaoka says that a vector bundle $E$ on a smooth complex projective curve $C$ is semistable if 
and only if the relative anticanonical line bundle for the
natural projection ${\mathbb P}(E) \, \longrightarrow\, C$ is nef 
\cite{Mi}. A holomorphic vector bundle $E$ on $C$ is semistable if and only if the vector bundle ${\rm ad}(E)\, 
\longrightarrow\, C$ of trace zero endomorphisms is numerically flat. The very useful notion of numerically flat 
vector bundles was introduced by Demailly, Peternell and Schneider in \cite{DPS}; we recall that a holomorphic 
vector bundle $V$ is numerically flat if both $V$ and $V^*$ are nef. Therefore, a reformulation of Miyaoka's theorem 
says that ${\rm ad}(E)$ is numerically flat if and only if the relative anticanonical line bundle on ${\mathbb 
P}(E)$ is nef.

For $E$ as above, fix any $1\,\leq\, r\, <\, {\rm rank}(E)$. Let ${\rm Gr}(r, E)\, \longrightarrow\, C$ be be the 
Grassmann bundle that parametrizes the $r$-dimensional quotients of the fibers of $E$. Bruzzo and Hern\'andez 
Ruip\'erez proved the following big generalization of the above theorem of Miyaoka: The relative anticanonical line 
bundle on ${\rm Gr}(r, E)$ is nef if and only if ${\rm ad}(E)$ is numerically flat \cite{BH}.

Let $X$ be a smooth complex projective variety and $E_G$ a holomorphic principal $G$--bundle on $X$, where $G$ is a 
simple affine algebraic group without center over $\mathbb C$. Fix a parabolic subgroup $P\, \subsetneq\, G$. In 
\cite{BB}, the following generalization of the above result of Bruzzo and Hern\'andez Ruip\'erez was proved: The 
adjoint vector bundle ${\rm ad}(E_G)$ is numerically flat if and only if the relative anticanonical line bundle for 
the natural projection $E_G/P \, \longrightarrow\, X$ is nef.

Our aim here is to investigate the principal bundles on a compact Fujiki manifold from the above point of view. We 
recall that a Fujiki manifold is a compact complex manifold which is the image of a bimeromorphic surjective map 
from a compact K\"ahler manifold \cite{Fu1}, \cite{Fu2}, or equivalently, the image of a holomorphic surjective map 
from a compact K\"ahler manifold \cite{Va}.

Let $G$ be a semisimple affine algebraic group, over $\mathbb C$, with one simple factor and $P\, \subsetneq\, G$
a fixed parabolic subgroup. Let $M$ be a compact connected Fujiki manifold and $E_G$ a holomorphic principal
$G$--bundle on $M$. The quotient map $E_G\, \longrightarrow\, E_G/P$ defines a holomorphic principal $P$--bundle
on $E_G/P$. The top exterior product of the adjoint bundle for this principal $P$--bundle is the relative
canonical bundle for the natural projection $E_G/P\, \longrightarrow\, M$.

We prove the following (see Theorem \ref{thm1} and Theorem \ref{thm2}):

\begin{theorem}\label{thmi}
Let $E_G$ be a holomorphic principal $G$--bundle on a compact connected Fujiki manifold $M$.
Then the following three statements are equivalent:
\begin{enumerate}
\item The holomorphic vector bundle ${\rm ad}(E_G)$ is numerically flat.

\item The relative anticanonical bundle for the natural projection $E_G/P\, \longrightarrow\, M$ is nef.

\item For every quadruple of the form $(Z,\, \omega,\, \gamma, E_P)$, where $Z$ is a reduced irreducible compact
complex analytic space equipped with a K\"ahler form $\omega$,
$$\gamma\, :\, Z\, \longrightarrow\, M$$
is a holomorphic map, and $E_P\, \subset\, \gamma^*E_G$ is a holomorphic reduction of structure group
of the principal $G$--bundle $\gamma^*E_G$ to the subgroup $P$, the inequality
$${\rm degree}({\rm ad}(E_P))\, \leq\, 0$$
holds.
\end{enumerate}
\end{theorem}

Note that the first statement of Theorem \ref{thmi} does not involve the parabolic subgroup $P$. Therefore,
Theorem \ref{thmi} has the following corollary:

\begin{corollary}\label{cori}
Let $E_G$ be a holomorphic principal $G$--bundle on a compact connected Fujiki manifold $M$.
If the second and third statements in Theorem \ref{thmi} hold for one parabolic subgroup $P\, \subsetneq\, G$,
then they hold for every proper parabolic subgroup of $G$.
\end{corollary}

\section{Numerically flat bundles}

Let $M$ be a compact connected complex manifold. Fix a Hermitian structure $H_0$ on $M$. Let $\omega_{H_0}$
be the corresponding positive $(1,\,1)$-form on $M$. A holomorphic line bundle $L$ on $M$ is called
\textit{numerically effective} (\textit{nef} for short) if for every $\epsilon\, >\, 0$,
there is a Hermitian structure $h_\epsilon$ on $L$ such that
$$
\Theta_{h_\epsilon}(L) \, \geq\, - \epsilon\cdot \omega_{H_0}\, ,
$$
where $\Theta_{h_\epsilon}(L)$ is the curvature of the Hermitian complex connection on $L$ corresponding to
$h_\epsilon$ \cite[p.~299, Definition 1.2]{DPS}. We note that while the definition of nefness uses $H_0$,
the nefness of any given line bundle is actually independent of the choice of $H_0$. 

A holomorphic vector bundle $V$ on $M$ is called nef if the tautological line bundle ${\mathcal O}_{{\mathbb P}(V)}(1)$
on ${\mathbb P}(V)$ is nef \cite[p.~305, Definition 1.9]{DPS}. A holomorphic vector bundle $V$ on $M$ is
called \textit{numerically flat} if both $V$ and its dual $V^*$ are nef \cite[p.~311, Definition 1.17]{DPS}.

A compact connected complex manifold $M$ is called a Fujiki manifold if there exists a
surjective bimeromorphic map
$$
f\,:\, Y\,\longrightarrow\, M
$$
where $Y$ is a compact K\"ahler manifold \cite{Fu1}, \cite{Fu2}. A basic theorem of Varouchas says that $M$ is a
Fujiki manifold if it is the image of a compact K\"ahler manifold by a surjective morphism (it need not
be a bimeromorphism) \cite[p.~51, Theorem 5]{Va}.

\begin{proposition}\label{prop1}
Let $M$ be a compact connected Fujiki manifold, and let
$$
f\,:\, Y\,\longrightarrow\, M
$$
be a surjective bimeromorphic map, where $Y$ is a compact connected K\"ahler manifold. Then a holomorphic
vector bundle $E$ on $M$ is numerically flat if and only if the pulled back vector bundle $f^*E$ is
numerically flat.
\end{proposition}

\begin{proof}
Given any nef vector bundle $F$ on $M$, the vector bundle $f^*F$ is nef \cite[p.~305, Proposition 1.10]{DPS}.
Applying this to both $E$ and $E^*$ we conclude that $f^*E$ is numerically flat if $E$ is numerically flat.

To prove the converse, assume that $f^*E$ is numerically flat. A structure theorem of \cite{DPS} says that
$f^*E$ admits a filtration of holomorphic subbundles
\begin{equation}\label{e1}
0\,=\, V_0\, \subset\, V_1\, \subset\, \cdots \, \subset\, V_{\ell-1}\, \subset\, V_\ell \,=\, f^*E
\end{equation}
such that the holomorphic vector bundle $V_i/V_{i-1}$ admits a unitary flat connection for all $1\, \leq\,
i\, \leq\, \ell$ (see \cite[p.~311, Theorem 1.18]{DPS}). This immediately implies the following:
\begin{enumerate}
\item the Chern class $c_i(f^*E)\, \in\, H^{2i}(Y,\, {\mathbb R})$ vanishes for every $i\, \geq\, 1$
(see \cite[p.~311, Corollary 1.19]{DPS}), and

\item the vector bundle $f^*E$ is pseudostable (see \cite[p.~23, Definition 2.1]{BG} for
pseudostable bundles; set the Higgs field $\theta$ in \cite[Definition 2.1]{BG} to be the zero section).
\end{enumerate}
Consequently, $f^*E$ admits a flat holomorphic connection $\widetilde{\nabla}$ such that
\begin{enumerate}
\item the connection $\widetilde{\nabla}$ preserves the subbundle $V_i$ in \eqref{e1} for every $1\, \leq\,
i\, \leq\, \ell$, and

\item the holomorphic connection on $V_i/V_{i-1}$ induced by $\widetilde{\nabla}$ is unitary flat for
every $1\, \leq\, i\, \leq\, \ell$.
\end{enumerate}
(See \cite[p.~20, Theorem 1.1]{BG}; set the Higgs field in \cite[Theorem 1.1]{BG} to be the zero section.)

For any $1\, \leq\, i\, \leq\, \ell$, let $\widetilde{\nabla}^i$ be the connection on $V_i$ induced by
$\widetilde{\nabla}$.

Since the map $f$ is a bimeromorphism, the induced homomorphism of fundamental groups
$$f_*\, :\, \pi_1(Y)\, \longrightarrow\,\pi_1(M)$$
is actually an isomorphism. Therefore, we conclude the following:
\begin{enumerate}
\item the flat connection $\widetilde{\nabla}$ on $f^*E$ is the pull-back of a flat connection $\nabla$ on $E$, and

\item the flat subbundle $(V_i,\, \widetilde{\nabla}^i)$ of $(f^*E,\, \widetilde{\nabla})$ descends to
a flat subbundle $(E_i,\, {\nabla}^i)$ of $(E,\, \nabla)$ for all $1\, \leq\, i\, \leq\, \ell$. In other
words, $(f^*E,\, \widetilde{\nabla})$ is the pullback of the flat subbundle $(E_i,\, {\nabla}^i)$
of $(E,\, \nabla)$.
\end{enumerate}

Furthermore, the connection on $E_i/E_{i-1}$ induced by $\nabla^i$ is unitary flat, because the connection
on $V_i/V_{i-1}$ induced by $\widetilde{\nabla}$ is unitary flat. The vector bundle $E_i/E_{i-1}$ is
numerically flat because it admits a unitary flat connection. Since the extension of a nef vector bundle
by a nef vector bundle is again nef \cite[p.~308, Propoisition 1.15(ii)]{DPS}, it follows immediately that
the extension of a numerically flat vector bundle by a numerically flat vector bundle is again numerically flat.
Consequently, the vector bundle $E$ is numerically flat.
\end{proof}

The following is a consequence of the proof of Proposition \ref{prop1}.

\begin{corollary}\label{cor1}
Let $E$ be a holomorphic vector bundle on a compact connected Fujiki manifold $M$. Then $E$ is numerically
flat if and only if there is a filtration of holomorphic subbundles of $E$
$$
0\,=\, E_0\, \subset\, E_1\, \subset\, \cdots \, \subset\, E_{\ell-1}\, \subset\, E_\ell \,=\, E
$$
and a flat holomorphic $\nabla$ on $E$ such that
\begin{enumerate}
\item $\nabla$ preserves $E_i$ for all $1\, \leq\, i\, \leq\, \ell$, and

\item the connection on $E_i/E_{i-1}$ induced by $\nabla$ is unitary for all $1\, \leq\, i\, \leq\, \ell$.
\end{enumerate}
\end{corollary}

\begin{proof}
If there is a filtration as above and a flat holomorphic $\nabla$ on $E$ satisfying the above two conditions,
then $E_i/E_{i-1}$ is numerically flat for all $1\, \leq\, i\, \leq\, \ell$. Therefore, using
\cite[p.~308, Propoisition 1.15(ii)]{DPS} it is deduced that $E$ is numerically flat.

To prove the converse assume that $E$ is numerically flat. Take any 
surjective bimeromorphic map
$$
f\,:\, Y\,\longrightarrow\, X
$$
as in Proposition \ref{prop1} with $Y$ K\"ahler. Then $f^*E$ is numerically flat by Proposition \ref{prop1}. As we 
saw in the proof of Proposition \ref{prop1}, this implies that there is a filtration as in the statement of the 
corollary and a flat holomorphic $\nabla$ on $E$ satisfying the two conditions in the statement of the corollary. 
\end{proof}

\section{Principal bundles on Fujiki manifolds}

Let $G$ be a connected complex semisimple affine algebraic group with one simple factor.
Fix a parabolic subgroup
$$
P\, \subsetneq\, G\, .
$$
The Lie algebras of $G$ and $P$ will be denoted by $\mathfrak g$ and $\mathfrak p$ respectively.

As before, $M$ is a compact connected Fujiki manifold. Let $E_G$ be a holomorphic principal $G$--bundle
on $M$. Let
\begin{equation}\label{e2}
\text{ad}(E_G)\, :=\, E_G\times^G \mathfrak g\, \longrightarrow\, M
\end{equation}
be the adjoint bundle for $E_G$; its fibers are Lie algebras identified with $\mathfrak g$ uniquely up to
conjugations.

Let $p\, :\, E_G \, \longrightarrow\, E_G/P$ and
\begin{equation}\label{e5}
\varphi\, :\, E_G/P\, \longrightarrow\, M
\end{equation}
be the natural projections. Consider the projection
$$
\varphi^*E_G\, :=\, (E_G/P)\times_M E_G\, \xrightarrow{\,{\rm Id}\times p\,\,}\, (E_G/P)\times_M (E_G/P)\, .
$$
Let
\begin{equation}\label{e4}
{\mathcal E}_P\, \subset\, \varphi^*E_G
\end{equation}
be the inverse image of the diagonal $E_G/P\, \subset\, (E_G/P)\times_M (E_G/P)$ under this
projection. It is straight-forward
to check that ${\mathcal E}_P$ is a holomorphic reduction of structure group of the principal $G$--bundle
$\varphi^*E_G$ to the subgroup $P\, \subset\, G$.

\begin{theorem}\label{thm1}
Let $E_G$ be a holomorphic principal $G$--bundle on a compact connected Fujiki
manifold $M$. Then the following two statements are equivalent:
\begin{enumerate}
\item The holomorphic vector bundle ${\rm ad}(E_G)$ in \eqref{e2} is numerically flat.

\item The holomorphic line bundle $\bigwedge^{\rm top} {\rm ad}({\mathcal E}_P)^*\, \longrightarrow\,
E_G/P$ is nef, where ${\mathcal E}_P$ in the holomorphic principal $P$--bundle constructed in \eqref{e4}.
\end{enumerate}
\end{theorem}

\begin{proof}
First assume that ${\rm ad}(E_G)$ is numerically flat.
We will show that $\bigwedge^{\rm top} {\rm ad}({\mathcal E}_P)^*$ is nef. For that the following
lemma will be used.

\begin{lemma}\label{lem1}
The direct image $\varphi_*\bigwedge^{\rm top} {\rm ad}({\mathcal E}_P)^*\, \longrightarrow\, M$, where
$\varphi$ is the projection in \eqref{e5}, is a vector bundle of positive rank. If this vector bundle $\varphi_*
\bigwedge^{\rm top} {\rm ad}({\mathcal E}_P)^*$ is nef, then $\bigwedge^{\rm top} {\rm ad}({\mathcal E}_P)^*$ is
also nef.
\end{lemma}

\begin{proof}[{Proof of Lemma \ref{lem1}}]
The line bundle
$$\bigwedge\nolimits^{\rm top} {\rm ad}({\mathcal E}_P)^*\, \longrightarrow\, E_G/P$$
is the relative anti-canonical line bundle for the projection $\varphi$. So $\bigwedge^{\rm top} {\rm ad}({\mathcal 
E}_P)^*$ is relatively ample (the anti-canonical line bundle of $G/P$ is ample). Therefore, $\bigwedge^{\rm top} 
{\rm ad}({\mathcal E}_P)^*$ is relatively very ample (an ample line bundle on $G/P$ is very ample; see \cite[Theorem 
6.5(2)]{Sn}, \cite{Se}). Also the higher direct images of $\bigwedge^{\rm top} {\rm ad}({\mathcal E}_P)^*$ vanish
by the Kodaira vanishing theorem.
These imply that $\varphi_*\bigwedge^{\rm top} {\rm ad}({\mathcal E}_P)^* \, \longrightarrow\, M$ is a holomorphic 
vector bundle of positive rank.

Since $\bigwedge^{\rm top} {\rm ad}({\mathcal E}_P)^*$ is relatively very ample, we get an embedding
$$
\eta\, :\, E_G/P\, \longrightarrow\, {\mathbb P}\left(\varphi_*\bigwedge\nolimits^{\rm top}
{\rm ad}({\mathcal E}_P)^*\right)\, .
$$
The pulled back line bundle $\eta^*{\mathcal O}_{{\mathbb P}(\varphi_*\bigwedge\nolimits^{\rm top} {\rm 
ad}({\mathcal E}_P)^*)}(1)$ is identified with $\bigwedge^{\rm top} {\rm ad}({\mathcal E}_P)^*$. Consequently, 
$\bigwedge^{\rm top} {\rm ad}({\mathcal E}_P)^*$ is nef if $\varphi_* \bigwedge^{\rm top} {\rm ad}({\mathcal 
E}_P)^*$ is nef, because ${\mathcal O}_{{\mathbb P}(\varphi_*\bigwedge\nolimits^{\rm top} {\rm ad}({\mathcal 
E}_P)^*)}(1)$ is nef in that case.
\end{proof}

Let
\begin{equation}\label{e6}
Z_G\, \subset\, G
\end{equation}
be the center of $G$. We note that the left-translation action of $G$ on $G/P$ produces an action of $G$ on
$H^0(G/P,\, K^{-1}_{G/P})$, where $K^{-1}_{G/P}$ is the anti-canonical line
bundle. The action of $Z_G$ on $G/P$ is trivial because $Z_G\, \subset\, P$. Also, the
action of $Z_G$ on $H^0(G/P,\, K^{-1}_{G/P})$ is trivial, because the adjoint action of
$Z_G$ is trivial. So we get an action of $G/Z_G$ on $H^0(G/P,\, K^{-1}_{G/P})$.

Since $G$ is semisimple, the adjoint action of $G/Z_G$ on $\mathfrak g$ is faithful. Also, the $G/Z_G$--module 
$\mathfrak g$ is isomorphic to ${\mathfrak g}^*$ using the Killing form on $\mathfrak g$. Therefore, there are 
nonnegative integers $t_1,\, \cdots,\, t_n$ such that the $G/Z_G$--module $H^0(G/P,\, K^{-1}_{G/P})$ is a direct 
summand of the $G/Z_G$--module
$$
\bigoplus_{j=1}^n {\mathfrak g}^{\otimes t_j}
$$
(see \cite[p.~40, Proposition 3.1(a)]{DM}). This implies that the holomorphic vector bundle
$\varphi_* \bigwedge^{\rm top} {\rm ad}({\mathcal E}_P)^*$ is a direct summand of the holomorphic vector bundle
$$
\bigoplus_{j=1}^n {\rm ad}(E_G)^{\otimes t_j}\, .
$$
Indeed, if $\bigoplus_{j=1}^n {\mathfrak g}^{\otimes t_j}\,=\, H^0(G/P,\, K^{-1}_{G/P})\oplus A$,
where $A$ is a $G/Z_G$--module, then
$$
\bigoplus_{j=1}^n {\rm ad}(E_G)^{\otimes t_j}\,=\, \left(\varphi_* \bigwedge\nolimits^{\rm top} {\rm ad}
({\mathcal E}_P)^*\right)\oplus \widetilde{A}\, ,
$$
where $\widetilde{A}\, \longrightarrow\, M$ is the holomorphic vector bundle associated to the principal
$G$--bundle $E_G$ for the $G$--module $A$ (any $G/Z_G$--module is also a $G$--module).

Since ${\rm ad}(E_G)$ is nef, the vector bundle ${\rm ad}(E_G)^{\otimes t_j}$ is nef \cite[p.~307, Proposition
1.14(i)]{DPS}, hence $\bigoplus_{j=1}^n {\rm ad}(E_G)^{\otimes t_j}$ is nef \cite[p.~308, Proposition
1.15(ii)]{DPS}, and therefore, its direct summand $\varphi_* \bigwedge^{\rm top} {\rm ad}({\mathcal E}_P)^*$
is nef \cite[p.~308, Proposition 1.15(i)]{DPS}. Now Lemma \ref{lem1} implies that
$\bigwedge^{\rm top} {\rm ad}({\mathcal E}_P)^*$ is nef.

To prove the converse, assume that $\bigwedge^{\rm top} {\rm ad}({\mathcal E}_P)^*$ is nef.
We will prove that ${\rm ad}(E_G)$ is numerically flat.

Since the holomorphic line bundle
$$
\bigwedge\nolimits^{\rm top} {\rm ad}({\mathcal E}_P)^*\,=\,
\left(\bigwedge\nolimits^{\rm top} {\rm ad}({\mathcal E}_P)^*\right)^{\otimes 2}\otimes
\bigwedge\nolimits^{\rm top} {\rm ad}({\mathcal E}_P)\,=\,
\left(\bigwedge\nolimits^{\rm top} {\rm ad}({\mathcal E}_P)^*\right)^{\otimes 2}\otimes K_\varphi
$$
is nef, where $K_\varphi$ is the relative canonical bundle for the projection $\varphi$ in \eqref{e5},
we conclude that the direct image $\varphi_* (\bigwedge\nolimits^{\rm top} {\rm ad}({\mathcal
E}_P)^*)^{\otimes 2}$ is a nef vector bundle \cite[p.~895, Th\'eor\`eme 2]{Mo}.

On the other hand, $\varphi_* (\bigwedge\nolimits^{\rm top} {\rm ad}({\mathcal E}_P)^*)^{\otimes 2}$ coincides with 
the holomorphic vector bundle on $M$ associated to the principal $G$--bundle $E_G$ for the $G$--module $H^0(G/P,\, 
(K^{-1}_{G/P})^{\otimes 2})$. Since the group $G$ is semisimple, it does not have any nontrivial character, in 
particular, $\bigwedge^{\rm top} H^0(G/P,\, (K^{-1}_{G/P})^{\otimes 2})$ is the trivial $G$--module. This implies 
that the associated holomorphic line bundle $\bigwedge^{\rm top} \left(\varphi_* (\bigwedge\nolimits^{\rm top} {\rm 
ad}({\mathcal E}_P)^*)^{\otimes 2}\right)$ is trivial. Since $\varphi_* (\bigwedge\nolimits^{\rm top} {\rm 
ad}({\mathcal E}_P)^*)^{\otimes 2}$ is nef, this implies that $\varphi_* (\bigwedge\nolimits^{\rm top} {\rm 
ad}({\mathcal E}_P)^*)^{\otimes 2}$ is numerically flat (see \cite[p.~311, Definition 1.17]{DPS}).
Since the tensor product of two nef bundles is nef \cite[p.~307, Proposition 1.14(i)]{DPS}, and both
$\varphi_* (\bigwedge\nolimits^{\rm top} {\rm ad}({\mathcal E}_P)^*)^{\otimes 2}$ and
$\varphi_* (\bigwedge\nolimits^{\rm top} {\rm ad}({\mathcal E}_P)^*)^{\otimes 2}$ are nef, we conclude that
that the vector bundle
\begin{equation}\label{e7}
\text{End}(\varphi_* (\bigwedge\nolimits^{\rm top} {\rm ad}({\mathcal E}_P)^*)^{\otimes 2})\,=\,
(\varphi_* (\bigwedge\nolimits^{\rm top} {\rm ad}({\mathcal E}_P)^*)^{\otimes 2})\otimes
(\varphi_* (\bigwedge\nolimits^{\rm top} {\rm ad}({\mathcal E}_P)^*)^{\otimes 2})^*
\end{equation}
is nef. As $\bigwedge\nolimits^{\rm top}\text{End}
(\varphi_* (\bigwedge\nolimits^{\rm top} {\rm ad}({\mathcal E}_P)^*)^{\otimes 2})$ is the trivial line bundle,
it now follows that $\text{End}(\varphi_* (\bigwedge\nolimits^{\rm top} {\rm ad}({\mathcal E}_P)^*)^{\otimes 2})$
is numerically flat (see \cite[p.~311, Definition 1.17]{DPS}).

Consider the $G/Z_G$--module $H^0(G/P,\, (K^{-1}_{G/P})^{\otimes 2})$ (the center $Z_G$ acts trivially on it). Let
$$
\rho\, :\, G/Z_G\, \longrightarrow\, \text{GL}(H^0(G/P,\, (K^{-1}_{G/P})^{\otimes 2}))
$$
be the corresponding homomorphism. We note that $G/Z_G$ is simple without center because $G$ has only one simple
factor. Hence the above homomorphism $\rho$ is injective. Therefore, the $G$--module $\text{Lie}(G/Z_G)
\,=\,\mathfrak g$ is a direct summand of the $G$--module
$$
\text{Lie}(\text{GL}(H^0(G/P,\, (K^{-1}_{G/P})^{\otimes 2})))\,=\, {\rm End}(H^0(G/P,\, (K^{-1}_{G/P})^{\otimes 2}))\,.
$$
On the other hand, the holomorphic vector bundle on $M$ associated to the principal $G$--bundle $E_G$ for the
$G$--module ${\rm End}(H^0(G/P,\, (K^{-1}_{G/P})^{\otimes 2}))$ coincides with the vector bundle in \eqref{e7}.
Since the $G$--module $\mathfrak g$ is a direct summand of
${\rm End}(H^0(G/P,\, (K^{-1}_{G/P})^{\otimes 2}))$,
we conclude that $\text{ad}(E_G)$ is a direct summand of the vector bundle in \eqref{e7}. We saw that the
vector bundle in \eqref{e7} is numerically flat. So its direct summand $\text{ad}(E_G)$ is also
numerically flat \cite[p.~308, Proposition 1.15(i)]{DPS}. This completes the proof of the theorem.
\end{proof}

\section{Pullback to K\"ahler manifolds}

Take $M$ and $E_G$ as before. Let
\begin{equation}\label{f1}
(Z,\, \omega)
\end{equation}
be a reduced irreducible compact complex analytic space $Z$ with a K\"ahler form $\omega$, and let
\begin{equation}\label{f2}
\gamma\, :\, Z\, \longrightarrow\, M
\end{equation}
be a holomorphic map. Consider the holomorphic principal $G$--bundle $\gamma^*E_G$ on $Z$.
Giving a holomorphic reduction of structure group
$$
E_P\, \subset\, \gamma^*E_G
$$
of the principal $G$--bundle $\gamma^*E_G$ to $P\, \subset\, G$ is equivalent to giving a holomorphic section
$$
\sigma\, :\, Z\, \longrightarrow\, (\gamma^*E_G)/P\, =\, \gamma^*(E_G/P)
$$
of the natural projection $(\gamma^*E_G)/P\, \longrightarrow\, Z$. Indeed, the inverse image
of $\sigma(Z)\, \subset\, (\gamma^*E_G)/P$ for the quotient map $\gamma^*E_G\, \longrightarrow\,
(\gamma^*E_G)/P$ is a holomorphic reduction of structure group of $\gamma^*E_G$ to $P$.

Take any holomorphic reduction of structure group $E_P\, \subset\, \gamma^*E_G$ of $\gamma^*E_G$ to $P$.
Let
\begin{equation}\label{e3}
\text{ad}(E_P)\, :=\, E_P\times^P \mathfrak p\, \longrightarrow\, Z
\end{equation}
be the adjoint bundle of $E_P$; the inclusion map of $\mathfrak p$ in $\mathfrak g$ produces a
map $$\text{ad}(E_P)\, \longrightarrow\, \text{ad}(\gamma^*E_G)\,=\,\gamma^*\text{ad}(E_G)\, ,$$
so $\text{ad}(E_P)$ is a subbundle of $\gamma^*\text{ad}(E_G)$.

For a holomorphic vector bundle $W$ on $Z$, define
$$
\text{degree}(W)\, :=\, \int_Z c_1(W)\wedge\omega^{d-1}\, \in\, \mathbb R\, ,
$$
where $d\,=\, \dim_{\mathbb C} Z$; see \cite[p.~168, (7.1)]{Ko}.

\begin{theorem}\label{thm2}
Let $E_G$ be a holomorphic principal $G$--bundle on a compact connected Fujiki manifold $M$. Then
the following two statements are equivalent:
\begin{enumerate}
\item The holomorphic vector bundle ${\rm ad}(E_G)$ in \eqref{e2} is numerically flat.

\item For every triple $(Z,\, \omega,\, \gamma)$ as in \eqref{f1} and \eqref{f2}, and
every holomorphic reduction of structure group $E_P\, \subset\, \gamma^*E_G$ of $\gamma^*E_G$ to
$P$, the inequality
$$
{\rm degree}({\rm ad}(E_P))\, \leq\, 0
$$
holds, where ${\rm ad}(E_P)$ is the adjoint bundle in \eqref{e3}.
\end{enumerate}
\end{theorem}

\begin{proof}
First assume that ${\rm ad}(E_G)$ is numerically flat. Take any $(Z,\, \omega,\, \gamma)$ as in
\eqref{f1} and \eqref{f2}. Since ${\rm ad}(E_G)$ is numerically flat, it follows that
$\gamma^*{\rm ad}(E_G)\,=\, {\rm ad}(\gamma^*E_G)$ is also numerically flat \cite[p.~305, Proposition 1.10]{DPS}.

Let
\begin{equation}\label{phi}
\phi\, :\, (\gamma^*E_G)/P\, \longrightarrow\, Z
\end{equation}
be the natural projection.
The holomorphic principal $P$--bundle $$\gamma^*E_G\, \longrightarrow\, (\gamma^*E_G)/P\,=\, \gamma^*(E_G/P)$$
over $(\gamma^*E_G)/P$ will be denoted by ${\mathcal F}_P$; it is a holomorphic reduction of structure group
of the principal $G$--bundle $\phi^*\gamma^*E_G$
to $P\, \subset\, G$. Since $\gamma^*{\rm ad}(E_G)$ is numerically flat, from Theorem \ref{thm1} we know
that the line bundle $\bigwedge^{\rm top} {\rm ad}({\mathcal F}_P)^*$ is nef.

Let $E_P\, \subset\, \gamma^*E_G$ be a holomorphic reduction of structure group of
the principal $G$--bundle $\gamma^*E_G$ to the subgroup $P$. It corresponds to
a section
$$
\beta\, :\, Z\, \longrightarrow\, (\gamma^*E_G)/P
$$
of the projection $\phi$ in \eqref{phi}; the holomorphic principal $P$--bundle $E_P$ is the pullback
$\beta^*{\mathcal F}_P$, where ${\mathcal F}_P$ is the principal $P$--bundle defined above. Therefore, we have
$$
\bigwedge\nolimits^{\rm top} {\rm ad}(E_P)\,=\, \beta^*\bigwedge\nolimits^{\rm top} {\rm ad}({\mathcal F}_P)\, .
$$
This, and the above observation that $\bigwedge^{\rm top} {\rm ad}({\mathcal F}_P)^*$ is nef, together
imply that $\bigwedge\nolimits^{\rm top} {\rm ad}(E_P)^*$ is nef. This immediately implies that
$$
{\rm degree}({\rm ad}(E_P))\, \leq\, 0\, .
$$

To prove that converse, assume that
$$
{\rm degree}({\rm ad}(E_P))\, \leq\, 0
$$
for every triple $(Z,\, \omega,\, \gamma)$ as in \eqref{f1} and \eqref{f2}, and
every holomorphic reduction of structure group $E_P\, \subset\, \gamma^*E_G$ of $\gamma^*E_G$ to $P$.
We will prove that ${\rm ad}(E_G)$ is numerically flat.

Take a surjective bimeromorphic map
\begin{equation}\label{ef}
f\,:\, Y\,\longrightarrow\, M\, ,
\end{equation}
where $Y$ is a compact connected K\"ahler manifold. From
Proposition \ref{prop1} we know that ${\rm ad}(E_G)$ is numerically flat if $f^*{\rm ad}(E_G)
\,=\, {\rm ad}(f^*E_G)$ is numerically flat.

Let
\begin{equation}\label{psi}
\psi\, :\, (f^*E_G)/P\,=\, f^*(E_G/P)\, \longrightarrow\, Y
\end{equation}
be the natural projection. The holomorphic principal $P$--bundle
\begin{equation}\label{tf}
f^*E_G\, \longrightarrow\, (f^*E_G)/P\,=\, f^*(E_G/P)
\end{equation}
will be denoted by $\widetilde{\mathcal F}_P$. Note that $\widetilde{\mathcal F}_P$ is a holomorphic
reduction of structure group of the principal $G$--bundle $(f\circ\psi)^*E_G\,=\, \psi^*f^*E_G$ to
$P\, \subset\, G$. To show that $f^*{\rm ad}(E_G)$ is numerically flat,
first note that Theorem \ref{thm1} says that it suffices to
prove that $\bigwedge^{\rm top} {\rm ad}(\widetilde{\mathcal F}_P)^*$ is nef. Now, to prove that
$\bigwedge^{\rm top} {\rm ad}(\widetilde{\mathcal F}_P)^*$ is nef, we will use the following
criterion of Demailly and Paun, \cite{DP}, for nefness.

A holomorphic line bundle $L$ on a compact K\"ahler manifold $N$ is nef if and only if
for every K\"ahler form $\omega_N$ on $N$, and every irreducible closed connected analytic subspace
${\mathcal S}\, \subset\, N$, the inequality
\begin{equation}\label{dp}
\int_{\mathcal S} c_1(L)\wedge \omega^{s-1}_N\, \geq\, 0
\end{equation}
holds, where $s\,=\, \dim_{\mathbb C}{\mathcal S}$ \cite[p.~1248, Corollary 0.4]{DP}.

Set $$(N,\, L)\,=\, ((f^*E_G)/P,\, \bigwedge^{\rm top} {\rm ad}(\widetilde{\mathcal F}_P)^*).$$ Take any
$(\omega_N,\, {\mathcal S})$ as above, so $${\mathcal S}\, \subset\,
(f^*E_G)/P\,.$$ Set $Z$ in the statement of the theorem to be ${\mathcal S}$, and
set $\gamma$ in the statement of the theorem to be the composition of maps
$$
{\mathcal S}\,\hookrightarrow\, (f^*E_G)/P\, \stackrel{\psi}{\longrightarrow}\, Y
\, \stackrel{f}{\longrightarrow}\, M\, ,
$$
where $\psi$ and $f$ are the maps in \eqref{psi} and \eqref{ef} respectively. Set the
reduction $E_P\,\subset\, \gamma^*E_G$ to be the restriction of the reduction
$$
\widetilde{\mathcal F}_P\, \subset\, (f\circ\psi)^*E_G
$$
(see \eqref{tf}) to ${\mathcal S}\,\subset\, (f^*E_G)/P$. So we have
$$
E_P\,=\,(\widetilde{\mathcal F}_P)\big\vert_{\mathcal S}\, .
$$
This implies that
$$
\bigwedge\nolimits^{\rm top} {\rm ad}(E_P)\,=\, \bigwedge\nolimits^{\rm top} {\rm ad}(\widetilde{\mathcal F}_P).
$$
Therefore, the given condition that
$$
{\rm degree}({\rm ad}(E_P))\, \leq\, 0
$$
implies that the inequality in \eqref{dp} holds. Now the above mentioned criterion of \cite{DP}
for nefness implies that $\bigwedge^{\rm top} {\rm ad}(\widetilde{\mathcal F}_P)^*$ is nef.
\end{proof}


\end{document}